\documentclass[reqno]{amsart}
\usepackage{latexsym, amsmath, amsfonts, amscd, amssymb, verbatim, mathtools, enumerate}
\usepackage{xcolor}
\allowdisplaybreaks
\usepackage{enumerate}
\usepackage{amsthm}
\usepackage{csquotes}
\MakeOuterQuote{"}

\normalsize
\newtheorem*{assumption*}{\assumptionnumber}
\providecommand{\assumptionnumber}{}


\newtheorem{theorem}{Theorem}[section]

\newtheorem{lemma}{Lemma}[section]

\newtheorem{commen}{Comment}[section]
\newtheorem{definition}{Definition}[section]
\newtheorem{remark}{Remark}[section]

\usepackage{mathtools}

\DeclareMathOperator*{\esssup}{ess\,sup}

\def\r{\mathbb{R}}
\begin{document}

\title[Exponential dichotomy and admissibility]{Exponential dichotomy and $(L^p,L^q)$-admissibility} 
         
\thanks{This research has been done under the research project QG.23.02 of Vietnam National University, Hanoi}

\subjclass[2010]{37D20, 37D25}

\keywords{Admissibility, Exponetial dichotomies with respect to a family of norms, $L^p$ spaces.}

\author{Trinh Viet Duoc}
\address{Trinh Viet Duoc,
Faculty of Mathematics, Mechanics, and Informatics\\ University of Science, Vietnam National University\\
334 Nguyen Trai, Hanoi, Vietnam; and 
 Thang Long Institute of Mathematics and Applied Sciences, Thang Long University, Nghiem Xuan Yem, Hanoi, Vietnam}
\email{tvduoc@gmail.com, duoctv@vnu.edu.vn}

\author{Nguyen Van Trong}
\address{Nguyen Van Trong, Faculty of Mathematics, Mechanics, and Informatics\\ University of Science, Vietnam National University\\
334 Nguyen Trai, Hanoi, Vietnam}
\email{mathchessguitar@gmail.com}


\begin{abstract}
We consider the notion of an exponential dichotomy with respect to a family of norms for an evolutionary family in a Banach space, and we characterize it by the admissibility of the pair $(L^p,L^q)$ for $p,q \in [1,\infty]$ with $p\ge q$. We then use this characterization to establish the robustness of an exponentially dichotomic evolutionary family with respect to a family of norms.
\end{abstract} 

\maketitle

	\section{Introduction}

It is well known that the characterization of uniform exponential dichotomy by admissibility has developed rapidly since Perron's paper \cite{11} in 1930. Many results on the case of differential equations can be found in the monographs due to Massera and Sch\"{a}ffer \cite{10}, Dalecskii and Krein \cite{4}, Levitan and Zhikov \cite{9}, Chicone and Latushkin \cite{3}. For nonuniform exponential dichotomy, we refer the reader to \cite{1,2} and references therein.

The notion of an exponential dichotomy with respect to a family of norms was introduced by Barreira et. al., in \cite{14,17} they characterized the exponential dichotomy with respect to a family of norms in terms of the admissibility of the pair $(L^p,L^q)$ for $p,q \in [1,\infty)$ with $p\ge q$ and $(L^\infty,L^\infty)$, respectively. In the case of the uniform exponential dichotomy, we refer the reader to \cite{16,12,13,15}. Note that a pair $(A, B)$ in this article corresponds to output-input.

For the evolutionary family on the half-line, Preda et. al. \cite{16} characterized the uniform exponential dichotomy in terms of the admissibility of the pair $(L^p,L^q)$ for $p,q \in [1,\infty]$ with $p\ge q$ and $(p,q) \ne (\infty, 1)$. Inspired by this result, our main aim is to characterize the exponential dichotomy with respect to a family of norms for a general evolutionary family on the line by the admissibility of the pair $(L^p,L^q)$ for $p,q \in [1,\infty]$ with $p\ge q$ and $(p,q) \ne (\infty, 1)$. The main contributions and novelty of our results are as follows:
\begin{enumerate}
\item The range of $(p,q)$ is larger, so our results cover some results in \cite{17,14}.
\item Better output space (contains more information, see $Y_1$ and Theorem \ref{thm:2.2}).
\item The family of dichotomic projections on the line is unique (see Comment \ref{re2}).
\end{enumerate}

\section{Preliminaries}
Let $X=(X,\Vert\cdot\Vert)$ be a real or complex Banach space and $B(X)$ be the space of all bounded linear operators on $X$. Moreover, we consider a family of norms $\Vert \cdot\Vert _t$ on $X$ for $t\in \r$ satisfying
\begin{enumerate}[(i)]
\item There exist constants $C>0$ and $\epsilon \geq 0$ such that 
\begin{equation}
\Vert x \Vert \leq \Vert x\Vert_t \leq Ce^{\epsilon|t|}\Vert x \Vert;    \label{2.3}
\end{equation}
\item The map $ t \mapsto \|x\|_t$ is measurable for each $x \in X$.   
\end{enumerate}
In $X$, we always consider the topology generated by the original norm $\|\cdot\|$. Extending the concept of an evolutionary family, we obtain the concept of an evolutionary family with respect to a family of norms $\Vert \cdot\Vert _t$ on $X$ for $t\in \r$ as follows.

\begin{definition} \label{dn:2.1}
 A family $\left( T(t,\tau)\right)_{t \geq \tau}$ with $t,\tau \in \mathbb{R}$ in $B(X)$ is said to be an evolutionary family with respect to a family of norms $\Vert \cdot\Vert _t$ with $t\in \r$ on the line if:
\begin{enumerate}[i.]
\item $T(t,t)=Id$  and $T(t,s)T(s,\tau)=T(t,\tau)$ for $t \geq s \geq \tau $ and $t,\tau,s \in \mathbb{R}$.
\item For each $x \in X $ and $\tau \in \r$,  the map  $t \mapsto T(t,\tau)x$  is continuous on $[\tau, \infty)$
and the map $s\mapsto T(\tau,s)x$ is continuous on $(-\infty,\tau]$.
\item There exist $K,a>0$ such that $\Vert T(t,\tau)x \Vert _t \leq Ke^{a(t-\tau)}\Vert x\Vert _{\tau}$ for all $x\in X$.
\end{enumerate}
\end{definition}

Similarly, we have the following definition.
\begin{definition}
An evolutionary family $\left( T(t,\tau)\right)_{t \geq \tau}$ with respect to a family of norms $\Vert \cdot\Vert _t$ on the line is said to be an exponential dichotomy with respect to the family of norms $\Vert \cdot\Vert _t$ if it satisfies the following conditions:
\begin{enumerate}[i.]
\item There exist dichotomic projections $P(t)$ in $B(X), t \in \mathbb{R}$ such that 
\begin{equation*}
P(t)T(t,\tau)=T(t,\tau)P(\tau) \quad \text{ for } t\geq \tau. 
\end{equation*}
\item The restriction mapping  
$T(t,\tau): Ker P(\tau) \rightarrow Ker P(t) $
is an isomorphism for $t \geq \tau$. We denote its inverse $(T(t,\tau)_{|_{\text{\rm Ker}P(\tau)}})^{-1}$ by $T(\tau,t)_|$.
\item There exist constants $a<0<b$ and $D>0$ such that
\begin{align}
\begin{split}
\Vert T(t,\tau)P(\tau)x \Vert _{\tau} &\leq De^{a(t-\tau)}\Vert x\Vert _{\tau},\\
\Vert T(\tau,t)_| Q(t)x \Vert _{t} &\leq De^{-b(t-\tau)}\Vert x\Vert _{t},
\end{split} \label{2.6}
\end{align}
for $t \geq \tau, x\in X$ and $Q(t)=I-P(t)$.
\end{enumerate}
\end{definition}
\begin{remark}\label{re1}
For each evolutionary family $\left( T(t,\tau)\right)_{t \geq \tau}$ is exponentially dichotomic with respect to a family of norms $\Vert \cdot\Vert _t$ on the line, we have $X=X^s(t) \oplus X^u(t)$ for each $t\in \r$, where $X^s(t)=P(t)X$ and $X^u(t)=Q(t)X$.
\end{remark}

We denote
$$ 
L^p(\r,X)=\{f:\r \to X \mid f \text{ is measurable and } \int_{-\infty}^{+\infty}\Vert f(t)\Vert_t^p \,dt < \infty\}
$$
for $p\in [1, \infty)$  and by
$$ 
L^\infty(\r,X)=\{f:\r \to X \mid f \text{ is measurable and } \esssup \Vert f(t)\Vert_t < \infty\}.
$$
Then, $L^p(\r,X)$ and $L^\infty (\r,X)$ are Banach spaces endowed with the respectively norms
 $$\Vert f\Vert_p= \Big( \int_{-\infty}^{+\infty}\Vert f(t)\Vert_t^p \,dt\Big)^{1/p} \quad
\text{ and } \quad  \Vert f\Vert_\infty = \esssup_{t\in \r} \|f(t)\|_t.
$$
For the simplicity of notations we will use the notation $L^p$ for $L^p(\r,X)$ where $p\in [1,\infty]$. Also, we set the Banach space
$$ C_b= \{f:\r \to X \mid f \text{ is continuous and } \sup \Vert f(t)\Vert_t < \infty\} \text{ with the norm }  \Vert f\Vert_\infty = \sup_{t\in \r} \|f(t)\|_t.$$
Finally, we note that the set $L^p$ is a Banach space with the norm $\Vert \cdot\Vert _p, p\in [1,\infty]$ which is proven similarly to \cite[Theorem 1]{14}.

 \section{Exponential dichotomy and ($L^p,L^q$)-admissibility}
 Let $p,q \in [1, \infty]$ with $p \geq q$. We characterize the notion of an exponential dichotomy for an evolutionary family with respect to a family of norms $\Vert \cdot\Vert _t$ by the admissibility of the Banach space pair $(Y_1,Y_2)$, where $Y_1=L^p\cap C_b$ and $Y_2=L^q$.

Given an evolutionary family $(T(t,\tau))_{t \geq \tau}$ on Banach space $X$ with respect to a family of norms $\Vert \cdot\Vert _t$ and a function $y: \r \to X$, 
we consider the integral equation
\begin{equation}
x(t)=T(t,\tau)x(\tau)+\displaystyle\int\limits_{\tau}^{t}T(t,s)y(s)ds  \quad  \text{for} \quad t\geq \tau. \label{2.7}
\end{equation}
The following lemma gives a sufficient condition for the continuity of the function $x$.
\begin{lemma}\label{lem1}
Given $y\in Y_2$. If a function $x: \r \to X$ satisfies \eqref{2.7}, then $x$ is continuous on $\r$.
\end{lemma}
\begin{proof}
Because the mapping $T(\cdot,\tau)x(\tau)$ is continuous on $[\tau,\infty)$,
one only have to prove the function
$$ z(t):= \int_\tau^t T(t,s) y(s)ds$$
is continuous on $[\tau,\infty)$ for each fixed $\tau\in \r$. Indeed, for $t_0\ge \tau$ and $t\in[t_0, t_0+1]$, we have
\begin{align*}
\|z(t)-z(t_0)\|&\le \int_{\tau}^{t_0}\| T(t,s)y(s)-T(t_0,s)y(s) \| ds 
+ \int_{t_0}^{t} \|T(t,s)y(s)\| ds\\ 
&\le   \int_{\tau}^{t_0}\| T(t,s)y(s)-T(t_0,s)y(s) \| ds 
+ M \int_{t_0}^{t} \|y(s)\|ds,\\
&\| T(t,s)y(s)-T(t_0,s)y(s) \| \le  2M \|y(s)\| \le 2M \|y(s)\|_s,
\end{align*}
where $M=\sup \{\Vert T(t,s)\Vert : s \in [\tau,t_0+1]\}< \infty$ (by the Banach--Steinhaus theorem).
By dominated convergence theorem and Lebesgue differentiation theorem (see \cite[Theorem 1.1.8 and Proposition 1.2.2, Chapt. 1]{ABHN}), we get
\[ 
\lim_{t\to t_0^+} \Big( \int_{\tau}^{t_0}\| T(t,s)y(s)-T(t_0,s)y(s) \| ds + M \int_{t_0}^{t} \|y(s)\|ds\Big) =0.
\]
Thus, $z$ is right continuous on $[\tau,\infty)$. For $t_0 > \tau$ and $t\in (\tau, t_0]$, we have
\begin{align*}
\|z(t)-z(t_0)\|&\le \int_{\tau}^{t}\| T(t,s)y(s)-T(t_0,s)y(s) \| ds + \int_{t}^{t_0} \|T(t_0,s)y(s)\| ds\\ 
&\le   \int_{\tau}^{t}\| T(t,s)y(s)-T(t_0,s)y(s) \| ds+M\int_{t}^{t_0} \|y(s)\|ds\\
&\le \int_{\tau}^{t_0}\| T(t,s)y(s)-T(t_0,s)y(s) \| ds+M\int_{t}^{t_0} \|y(s)\|ds.
\end{align*}
Similarly, $z$ is left continuous on $(\tau,\infty)$. 
Therefore, $z$ is continuous on $[\tau,\infty)$. Since $\tau$ is arbitrary, $x$ is continuous on $\r$. 
\end{proof}
We then define an operator related to the integral equation \eqref{2.7} as follows.
\begin{definition}\label{H}
The operator $H:\mathcal{D}(H) \to Y_2$ is defined by 
\begin{align*}
\begin{array}{rll}
\mathcal{D}(H)&:=&\{ x\in Y_1: \text{ there is } y \in Y_2 \text{ such that } x \text{ and } y \text { satisfying } \eqref{2.7}\},\\
Hx&:=&y.
\end{array}
\end{align*}
\end{definition}
We now justify the definition of $H$ and give its properties through the following lemma.
\begin{lemma} \label{lem2}
$H$ is a well-defined closed linear operator.
\end{lemma}
\begin{proof}
Let  $y_1,y_2 \in Y_2$ such that
 $$x(t)=T(t,\tau)x(\tau)+\displaystyle\int\limits_{\tau}^{t}T(t,s)y_1(s)ds,$$
 and
 $$x(t)=T(t,\tau)x(\tau)+\displaystyle\int\limits_{\tau}^{t}T(t,s)y_2(s)ds,$$
 for $t > \tau$. Then,
 $$\dfrac{t}{t-\tau}\displaystyle\int\limits_{\tau}^{t}T(t,s)y_1(s)ds=\frac{t}{t-\tau}\displaystyle\int\limits_{\tau}^{t}T(t,s)y_2(s)ds.$$
 Since the map  $s \mapsto T(t,s)y_i(s)$ is locally integrable for $i=1,2$, let $\tau \rightarrow t$, from the Lebesgue differentiation theorem it follows that $y_1(t)=y_2(t)$ for almost all $t \in \mathbb{R}$. Therefore, $H$ is well-defined. Furthermore, $H$ is obviously a linear operator.

Let $(x_n)_n \subset \mathcal{D}(H)$ such that $x_n $ converges to $x$ in $Y_1$ and $y_n:=Hx_n $ converges to $y$ in $Y_2$. 
For each fixed $\tau \in \r$, we have
 \begin{align*}
 x(t)-T(t,\tau)x(\tau)&= \lim_{n \rightarrow +\infty}\left( x_{n}(t)-T(t,\tau)x_{n}(\tau)\right)\\
 &= \lim_{n \rightarrow +\infty}\displaystyle\int\limits_{\tau}^{t}T(t,s)y_{n}(s)ds
 \end{align*}
 for $ t\geq \tau$. Moreover,
\begin{align*}
 \Big\Vert \displaystyle\int\limits_{\tau}^{t}T(t,s)y_{n}(s)ds&-\displaystyle\int\limits_{\tau}^{t}T(t,s)y(s)ds \Big\Vert \leq M \displaystyle\int\limits_{\tau}^{t}\Vert y_{n}(s)-y(s)\Vert ds\\
 &\leq M\left( \displaystyle\int\limits_{\tau}^{t}\Vert y_{n}(s)-y(s)\Vert ^q ds \right)^{\frac{1}{q}}(t-\tau)^{1-\frac{1}{q}}\\
 &\leq M
\Vert y_{n}-y \Vert _q (t-\tau)^{1-\frac{1}{q}},
\end{align*} 
where $M=\sup \{\Vert T(t,s)\Vert : s \in [\tau,t]\}< \infty$. Since $y_{n}$ converges to $y$ in $Y_2$, we obtain 
 $$ \displaystyle \lim_{n \rightarrow +\infty}\displaystyle\int\limits_{\tau}^{t}T(t,s)y_{n}(s)ds=\displaystyle\int\limits_{\tau}^{t}T(t,s)y(s)ds.$$ 
 Hence,
 $$x(t)-T(t,\tau)x(\tau)=\displaystyle\int\limits_{\tau}^{t}T(t,s)y(s)ds $$
for $t\geq \tau$. So, $x \in \mathcal{D}(H)$ and $Hx=y$. Therefore, $H$ is a closed operator.
\end{proof}

Firstly, we show that the exponential dichotomy with respect to a family of norms of an evolutionary family deduces $(L^p, L^q)$-admissibility (more precisely, $(L^p\cap C_b, L^q)$-admissibility), namely, $H: \mathcal{D}(H) \to Y_2$ is bijective.
\begin{theorem}\label{thm:2.2}
If the evolutionary family $(T(t,\tau))_{t \geq \tau}$ is exponentially dichotomic with respect to a family of norms $\Vert \cdot\Vert _t$, then for each $y \in Y_2$ there exists a unique $x \in Y_1$ satisfying \eqref{2.7}.
\end{theorem}
\begin{proof}
For $t\in \mathbb{R}$, we define
$$ x_1(t)=\displaystyle\int\limits_{-\infty}^{t}T(t,\tau)P(\tau)y(\tau)d\tau \quad \text{and}\quad  x_2(t)=\displaystyle\int\limits_{t}^{\infty}T(t,\tau)_|Q(\tau)y(\tau)d\tau.$$
Using Holder's inequality and \eqref{2.6}, we have
\begin{align*}
\|x_1(t)\| &\le \Vert x_1(t)\Vert _t  \le \displaystyle \int_{-\infty}^{t} \Vert T(t,\tau)P(\tau)y(\tau)\Vert _t d\tau  
\le \displaystyle \int_{-\infty}^{t} D e^{{a}(t-\tau )}\Vert y(\tau)\Vert _{\tau} d\tau \nonumber \\
&= \displaystyle\sum_{m=0}^{\infty} \displaystyle \int_{t-m-1}^{t-m} D e^{{a}(t-\tau )}\Vert y(\tau)\Vert _{\tau} d\tau 
\leq \displaystyle\sum_{m=0}^{\infty}D e^{{a}m} \displaystyle \int_{t-m-1}^{t-m} \Vert y(\tau)\Vert _{\tau} d\tau \nonumber  \\
&\le \displaystyle\sum_{m=0}^{\infty}D e^{{a}m} \left( \displaystyle \int_{t-m-1}^{t-m} \Vert y(\tau)\Vert _{\tau}^q d\tau \right)^{1/q} 
\leq \displaystyle\sum_{m=0}^{\infty}D e^{{a}m} \Vert y\Vert_q 
=\frac{D}{1-e^{a}}\Vert y\Vert _q, \\
\|x_2(t)\| &\le \Vert x_2(t)\Vert _t  \le \displaystyle \int_{t}^{+\infty} \Vert T(t,\tau)_|Q(\tau)y(\tau)\Vert _t d\tau  
\le \displaystyle \int_{t}^{+\infty} D e^{-{b}(\tau -t)}\Vert y(\tau)\Vert _{\tau} d\tau \nonumber \\
&= \displaystyle\sum_{m=0}^{\infty} \displaystyle \int_{t+m}^{t+m+1} D e^{-{b}(\tau -t)}\Vert y(\tau)\Vert _{\tau} d\tau 
\leq \displaystyle\sum_{m=0}^{\infty}D e^{-{b}m} \displaystyle \int_{t+m}^{t+m+1} \Vert y(\tau)\Vert _{\tau} d\tau \nonumber  \\
&\le \displaystyle\sum_{m=0}^{\infty}D e^{-{b}m} \left( \displaystyle \int_{t+m}^{t+m+1} \Vert y(\tau)\Vert _{\tau}^q d\tau \right)^{1/q} 
\leq \displaystyle\sum_{m=0}^{\infty}D e^{-{b}m} \Vert y\Vert_q 
=\frac{D}{1-e^{-{b}}}\Vert y\Vert _q. \notag
\end{align*}
Therefore, $x_1$ and $x_2$ are well-defined. Furthermore,
\begin{equation}\label{eq1}
 \|x_1(t)-x_2(t)\|_t \le \Big(\frac{D}{1-e^{a}} + \frac{D}{1-e^{-{b}}}\Big) \Vert y\Vert _q \quad \text{for all }\, t\in \r.
\end{equation}
Acting $P(t)$ to the both sides of \eqref{2.7}, we have
$$ P(t)x(t)= T(t,\tau)P(\tau)x(\tau)+\displaystyle\int\limits_{\tau}^{t}T(t,s)y(s)ds  \quad  \text{for} \quad t\geq \tau.$$
For $\tau \to -\infty$, we obtain
$$ \lim_{\tau \to -\infty}T(t,\tau)P(\tau)x(\tau) =  P(t)x(t) - x_1(t).$$
We have
$$\|T(t,\tau)P(\tau)x(\tau)\| \le \|T(t,\tau)P(\tau)x(\tau)\|_{t} \le D e^{a(t-\tau)} \|x(\tau)\|_{\tau} \le D e^{a(t-\tau)} \|x\|_{\infty}.$$
This leads to $P(t)x(t) = x_1(t)$ for each $t \in \mathbb{R}$.

Fix $\tau \in \mathbb{R} $, applying $Q(t)$ to the both sides of \eqref{2.7}, we have
$$ Q(t)x(t)= T(t,\tau)Q(\tau)x(\tau)+\displaystyle\int\limits_{\tau}^{t}T(t,s)Q(s)y(s)ds  \quad  \text{for} \quad t\geq \tau.$$
By $T(t,\tau): Q(\tau)X \to Q(t)X$ is an isomorphism, 
$$ T(\tau,t)_| Q(t)x(t) = Q(\tau)x(\tau)+\displaystyle\int\limits_{\tau}^{t}T(\tau,s)_|Q(s)y(s)ds \quad  \text{for} \quad t\geq \tau.$$
For $t \to +\infty$, we obtain
$$ \lim_{t\to +\infty}T(\tau,t)_| Q(t)x(t) = Q(\tau)x(\tau)+ x_2(\tau).$$
Same as above, we get $Q(\tau)x(\tau)= -x_2(\tau)$ for each $\tau \in \mathbb{R}$.

So, we have $x(t)=x_1(t)-x_2(t)$ for all $t\in \mathbb{R}$. Therefore, $x\in Y_1$ satisfying \eqref{2.7} is unique.
Next, we show that the functions  $x_1$ and $x_2$ belong to $L^p$.
We only consider $x_2$ since the argument for $x_1$ is analogous. 
Since $p \geq q$, there exists $r \in [1,\infty]$ such that
$$\dfrac{1}{r}+\dfrac{1}{q}=1+\dfrac{1}{p}.$$ 
Let
$$
 \omega (t)=
 \begin{cases}
 0& \text{if}\quad t \geq 0,\\
e^{{b}t}& \text{if}\quad t<0\\ 
 \end{cases}
 $$
 and $z(t)=\Vert y(t)\Vert_t$. We have $\omega \in L^{r}(\r)$ and $z\in L^q(\mathbb{R})$.
By Young's inequality  for convolution,   $w*z \in L^p(\mathbb{R})$. Furthermore,
$$(\omega *z)(t)= \displaystyle \int _{-\infty}^{+\infty}\omega (t-\tau)z(\tau)d\tau=\displaystyle \int _{t}^{+\infty}e^{-{b}(\tau -t)}\Vert y(\tau) \Vert_{\tau}d\tau.$$
We also have $\Vert x_2(t)\Vert _t \le D(\omega *z)(t)$ for all $t$. Thus, $x_2 \in L^p$ and
$$\Vert x_2\Vert_p \leq D \Vert \omega *z\Vert _p \le \frac{D}{(br)^{\frac{1}{r}}} \|y\|_q.$$
Similarly, $x_1 \in L^p$. Therefore, $x_1 - x_2 \in L^p$ and
$$ \Vert x_1 -x_2 \Vert_p \le \Big(\frac{D}{(-ar)^{\frac{1}{r}}}+ \frac{D}{(br)^{\frac{1}{r}}}\Big) \|y\|_q.$$
Let $x(t)=x_1(t)-x_2(t)$ for $t \in \mathbb{R}$, and take $t_0 \in \mathbb{R}$. We now show that $x$ satisfies \eqref{2.7} for $t\ge t_0$. 
\begin{align}
x(t)&=\displaystyle\int\limits_{t_0}^{t}T(t,\tau)y(\tau)d\tau-\displaystyle\int\limits_{t_0}^{t}T(t,\tau)P(\tau)y(\tau)d\tau -\displaystyle\int\limits_{t_0}^{t}T(t,\tau)Q(\tau)y(\tau)d\tau \nonumber
\\&+\displaystyle\int\limits_{-\infty}^{t}T(t,\tau)P(\tau)y(\tau)d\tau-\displaystyle\int\limits_{t}^{+\infty}T(t,\tau)_|Q(\tau)y(\tau)d\tau \nonumber\\
&=\displaystyle\int\limits_{t_0}^{t}T(t,\tau)y(\tau)d\tau+\displaystyle\int\limits_{-\infty}^{t_0}T(t,t_0)T(t_0,\tau)P(\tau)y(\tau)d\tau-\displaystyle\int\limits_{t_0}^{+\infty}T(t,t_0)T(t_0,\tau)_|Q(\tau)y(\tau)d\tau \nonumber\\
&=\displaystyle\int\limits_{t_0}^{t}T(t,\tau)y(\tau)d\tau+T(t,t_0)x(t_0).\nonumber 
\end{align}
By Lemma \ref{lem1} and \eqref{eq1}, $x\in C_b$.
So, there exists a unique $x \in Y_1$ satisfying \eqref{2.7}.
\end{proof}
 Now we establish the converse of Theorem \ref{thm:2.2}.
  \begin{theorem}\label{thm:2.3}
Let $p,q \in [1, \infty]$ with $p \geq q$ and $(p,q)\ne (\infty,1)$. 
  Assume that for each $y \in Y_2$ there exists a unique $x \in Y_1$ satisfying \eqref{2.7}. Then, the evolutionary family $(T(t,\tau))_{t \geq \tau}$ is exponentially dichotomic with respect to the family of norms $\Vert \cdot\Vert _t$.
  \end{theorem}
\begin{proof}
From the assumption of the theorem and Lemma \ref{lem2}, the inverse of $H$, $G:Y_2 \rightarrow \mathcal{D}(H) \subset Y_1$, is a closed linear operator. By the closed graph theorem, $G$ is bounded. 

For $\tau \in \mathbb{R}$, we define the following sets
\begin{align*}
F^{s}_{\tau}&=\{ x\in X :  \sup_{t\ge \tau}\|T(t,\tau)x\|_t <\infty \; \text{ and }\; z\in L^p \}, \; \text{where }\; z(t)=
\begin{cases}
T(t,\tau)x,& t\ge \tau,\\
0,&t<\tau.
\end{cases} \\ 
\mathcal{F}_{\tau}&=\left\{ \varphi: (-\infty, \tau] \to X  \; \text{ such that } \; \varphi(t)=T(t,s)\varphi(s) \; \text{ for all } \; s\le \tau \; \text{ and } \; s\le t\le \tau \right\}. \cr
F^{u}_{\tau}&=\{ x\in X :  \; \exists \; \varphi_x \in \mathcal{F}_{\tau}  \text{ such that }  \varphi_x(\tau)=x, \,   
\sup_{t \le \tau}\|\varphi_x(t)\|_t <\infty  \text{ and } w\in L^p\},\label{fu} \\
\text{ where }&
 w(t)=
\begin{cases}
\varphi_x(t),& t\le \tau,\\
0,&t>\tau.
\end{cases} \notag
\end{align*}
Clearly, $F^s_{\tau}$ and $ F_{\tau}^u$ are subspaces of $X$. For easy tracking, we divide the proof steps into the lemmas below.
\begin{lemma} \label{lem}
For each $\tau \in \r$,
\begin{equation}
X=F_{\tau}^s \oplus F_{\tau}^u. \label{2.15}
\end{equation}
\end{lemma}
\begin{proof}
For $x\in X$, take
$$
 g(t)=
 \begin{cases}
 T(t,\tau)x& \text{if}\quad t \in [\tau ,\tau +1],\\
 0 & \text{otherwise}. 
 \end{cases}
 $$
 It is easy to see that $g \in Y_2$. Hence, there exists $v \in \mathcal{D}(H)$ such that $Hv=g$. Furthermore,
 \begin{align*}
 v(t)&=T(t,\tau)v(\tau)+\displaystyle\int\limits_{\tau}^{t}T(t,s)g(s)ds\\
 &=T(t,\tau)v(\tau)+\displaystyle\int\limits_{\tau}^{t}T(t,s)T(s,\tau)\chi_{[\tau, \tau+1]} x ds\\
 &=T(t,\tau)v(\tau)+\displaystyle\int\limits_{\tau}^{t}T(t,\tau)\chi_{[\tau, \tau+1]}x ds\\
&=\begin{cases}
T(t,\tau)(v(\tau)+(t-\tau)x)& \text{ for } t\in [\tau, \tau+1),\\
T(t,\tau)(v(\tau)+x)& \text{ for } t\ge \tau+1.  
\end{cases}
 \end{align*}
Since $v \in Y_1$ leads to $v(\tau)+x \in F^s_{\tau}$. For $s \leq \tau$ and $t \in [s,\tau]$, we have
$$v(t)= T(t,s)v(s)+\displaystyle\int\limits_{s}^{t}T(t,\xi)g(\xi) d\xi=T(t,s)v(s).$$
Thus, $v_{|(-\infty,\tau]} \in \mathcal{F}_{\tau}$. From $v \in Y_1$ it follows that $v(\tau) \in F^u_{\tau}$.
Therefore, $x \in F^s_{\tau}+ F^u_{\tau}$ for all $x\in X$. So, $X=F^s_{\tau}+ F^u_{\tau}$.
Now we prove that if $x \in F^s_{\tau}\bigcap F^u_{\tau}$ then $x=0$. Indeed, we define $u:\mathbb{R} \rightarrow X$ as follows
 $$
 u(t)=
 \begin{cases}
 T(t,\tau)x & \text{if}\quad t\geq \tau,\\
\varphi_x(t) & \text{ otherwise}.
 \end{cases}
 $$
From the definition of $F^s_{\tau}$ and $F^u_{\tau}$ it follows that $u \in Y_1$. Moreover, we have $Hu=0$. 
Therefore, $u \in \mathcal{D}(H)$. Since $H$ is invertible, we conclude that  $u=0$ and therefore $x=0$.
\end{proof}

From the decomposition in \eqref{2.15}, we define the projections
$$ P(\tau):X \rightarrow F^s_{\tau} \quad \text{and} \quad  Q(\tau):X\rightarrow F^u_{\tau}$$
for each $\tau \in \r$. Then,  $P(\tau)+Q(\tau)=Id$.
 \begin{lemma}
There exits $M>0$ such that
\begin{equation}
\Vert P(\tau)x\Vert _\tau \leq M\Vert x\Vert_\tau   \label{2.16}
\end{equation} 
for $x \in X$ and $\tau \in \mathbb{R}$. Consequently, $P(\tau)$ is bounded for every $\tau \in \r$.         
\end{lemma}
\begin{proof}
By  $P(\tau)x  =v(\tau)+x$ (using the notation of Lemma \ref{lem}), we have
 $$\Vert P(\tau)x\Vert _{\tau}=\Vert v(\tau)+x \Vert _{\tau} \leq \Vert v(\tau)\Vert _{\tau}+\Vert x\Vert _{\tau}.$$ 
Moreover, $v(\tau)=T(\tau,s)v(s)$ for $s \in [\tau -1, \tau]$. Thus, 
$$ \Vert v(\tau)\Vert _{\tau}=\Vert T(\tau,s)v(s)\Vert _{\tau}
 \leq Ke^{c(\tau-s)}\Vert v(s)\Vert _s$$ for $s\in [\tau -1, \tau]$. Integrating over $s$, we obtain
 \begin{align*}
\Vert v(\tau)\Vert _{\tau}
 &\leq Ke^{c}\displaystyle\int_{\tau -1}^{\tau}\Vert v(s)\Vert _s ds
 \leq  Ke^{c}\left(\displaystyle\int_{\tau -1}^{\tau}\Vert v(s)\Vert _s^p ds\right)^{\frac{1}{p}}\\
 &\leq Ke^{c} \Vert v\Vert _p=Ke^c\Vert Gg\Vert _p
 \leq Ke^c \Vert G\Vert \Vert g\Vert _q\\
&=Ke^c\Vert G\Vert \left(\displaystyle\int_{\tau }^{\tau +1}\Vert T(t,\tau)x\Vert _{t}^q dt\right)^{\frac{1}{q}} 
\leq K^2e^{2c} \|G\| \Vert x \Vert _{\tau}.
\end{align*}
This show that \eqref{2.16} holds with $M=K^2e^{2c}\Vert G\Vert +1$.
\end{proof}
The spaces $F_{\tau}^s$ and $F_{\tau}^u$ are invariant with the evolutionary family.
\begin{lemma}\label{lem3}
$T(t,\tau)F_{\tau}^s \subset F_{t}^s$ and $T(t,\tau)F_{\tau}^u = F_{t}^u$ for all $t\ge \tau$.
\end{lemma}
\begin{proof}
Take $x\in F^{s}_{\tau}$ and with $\xi \ge t$, we have $T(\xi, t)T(t,\tau)x = T(\xi, \tau)x$. Thus,
$$ \sup_{\xi \ge t}\|T(\xi, t)T(t,\tau)x\|_{\xi} = \sup_{\xi \ge t} \|T(\xi, \tau)x\|_{\xi} \le \sup_{\xi \ge \tau} \|T(\xi, \tau)x\|_{\xi} <\infty, $$
and
$$ 
z_1(\xi)=
\begin{cases}
T(\xi, \tau)x,& \xi \ge t,\\
0,&\xi < t
\end{cases}
 $$
belongs to $L^p$.
Therefore, $T(t,\tau)x \in F^{s}_t$.
So, $ T(t,\tau) F^{s}_{\tau} \subset  F^{s}_t$ for all $t\ge \tau$.

Take $x\in F^{u}_{\tau}$, put $y=T(t,\tau)x$. We define the function $\varphi_y$ as follows
$$ 
\varphi_y(\xi)=
\begin{cases}
\varphi_x(\xi),& \xi \le \tau,\\
T(\xi,\tau)x,& \tau \le \xi \le t.
\end{cases}
 $$
Then, $\varphi_{y}\in \mathcal{F}_{t}$. Indeed, we easily see  $\varphi_y(\xi)=T(\xi, s)\varphi_{y}(s)$ for $s\ge \tau$ and $\xi \in [s,t]$, $s<\tau$ and $\xi \in[s,\tau]$.
For  $s<\tau$ and $\xi \in [\tau, t]$, we have
$$\varphi_y(\xi)=T(\xi, \tau)x = T(\xi, \tau)T(\tau, s)\varphi_{x}(s)=T(\xi,s)\varphi_x(s)=T(\xi, s)\varphi_{y}(s).$$
For $\xi \in [\tau, t]$, we have 
$$ \|T(\xi,\tau)x\|_{\xi} \le K e^{c(t-\tau)}\|x\|_{\tau}. $$
Therefore, $\sup_{\xi \le t}\|\varphi_y(\xi)\|_{\xi} <\infty$; and
$w_1 \in L^p$ with $w_1(\xi)=\varphi_y(\xi)$ if $\xi \le t$ and $w_1(\xi)=0$ if $\xi >t$. So, $T(t,\tau)x \in F^{u}_t$.
Therefore, $T(t,\tau) F^{u}_{\tau} \subset F^{u}_t$.

For $y\in F^{u}_t$, we easily check $x=\varphi_{y}(\tau) \in F^{u}_{\tau}$ with $\varphi_x := \varphi_{y}|_{(-\infty, \tau]}$. 
Because of $\varphi_y \in \mathcal{F}_t$,  $T(t, \tau)x = y$.
Therefore, $T(t,\tau) F^{u}_{\tau} = F^{u}_t$.
\end{proof}
By a similar proof as in \cite[Lemma 8]{14}, we obtain the exponential decay of the evolutionary family on the space $F_{\tau}^s$.
\begin{lemma}
There exist $\lambda , D>0 $ such that
\begin{equation*}
\Vert T(t,\tau)P(\tau)x\Vert_t \leq De^{-\lambda(t-\tau)}\Vert x\Vert_\tau  
\end{equation*}
 for $x \in X$ and $t \geq \tau$.
\end{lemma}
We now show that the evolutionary family is exponentially growing on the space $F_{\tau}^u$.
\begin{lemma}\label{lem4}
There exist $\lambda , D>0 $ such that 
\begin{equation}
\Vert T(t,\tau)Q(\tau)x\Vert _t \geq \frac{1}{D} e^{\lambda(t-\tau)}\Vert Q(\tau)x\Vert _\tau \label{2.21}
\end{equation}
 for $x \in X$ and $t \geq \tau$.
\end{lemma}
\begin{proof}
Firstly, we prove that there exists $C>0$ such that 
\begin{equation}
C\Vert T(t,\tau)x \Vert _t \geq \Vert x\Vert _{\tau} \label{2.22}
\end{equation}
 for $t\geq \tau, \,x \in F^u_{\tau}$. By Lemma \ref{lem3}, we have $y=T(t+1,\tau)x \in F^u_{t+1}$, $\varphi_y \in \mathcal{F}_{t+1}$, and $w\in L^p$, where
$$ 
w(\xi)=
\begin{cases}
\varphi_y(\xi),& \xi \le t+1,\\
0,& \xi > t+1,
\end{cases}
\quad \text{ and } \quad
\varphi_y (\xi)=
\begin{cases}
\varphi_x(\xi),& \xi \le \tau,\\
T(\xi, \tau)x,& \tau \le \xi  \le t+1.
\end{cases}
 $$
Let $\psi:\r\to [0,1]$  be a continuously differentiable function such that
$$ {\rm supp \,} \psi \subset (-\infty, t+1],\quad \psi(\xi)=1 \text{ \; for \; } \xi \le t,\quad \sup_{\xi \in \r}|\psi'(\xi)| \le 2. $$
It easy to verify that $\psi w \in Y_1$, $\psi' w \in Y_2$, and $H(\psi w)=\psi' w$. We have
 \begin{align*}
\Vert \psi w \Vert _{\infty}&=\Vert G(\psi' w )\Vert _{\infty} \leq \Vert G \Vert \Vert\psi' w \Vert _{q}
=\Vert G \Vert \left( \displaystyle \int_{t}^{t+1} \Vert \psi'(\xi) w(\xi) \Vert ^q_{\xi} d\xi\right)^{\frac{1}{q}}\\
&\leq 2 \Vert G \Vert \left( \displaystyle \int_{t}^{t+1} \Vert T(\xi ,t)T(t,\tau) x \Vert ^q_{\xi} d\xi\right)^{\frac{1}{q}}
\leq 2Ke^c \Vert G \Vert \Vert T(t,\tau)x\Vert _t.
\end{align*} 
Therefore, $\|x\|_{\tau}=\|(\psi w)(\tau)\|_{\tau} \le \|\psi w\|_{\infty}  \le 2Ke^c \Vert G \Vert \Vert T(t,\tau)x\Vert _t$.
We now show that there exists $T>0$ such that 
\begin{equation}
\Vert T(t,\tau)x\Vert _t \geq 2 \Vert x\Vert _{\tau} \quad  \text{ for } \quad t-\tau \geq T.  \label{2.23}
\end{equation}
Indeed, we consider the following functions
$$ y_2 (\xi)=- \chi _{[\tau ,t]}(\xi) w (\xi) \quad\text{and}\quad  v_2(\xi)=\int_{\xi}^{+\infty}\chi _{[\tau ,t]}(s)ds\, w(\xi).$$
Then, $v_2\in Y_1, y_2 \in Y_2$, and $Hv_2=y_2$. Therefore,
\begin{align*}
 \Vert v_2 \Vert _{\infty}&=\Vert G y_2 \Vert _{\infty} \leq \Vert G \Vert \Vert y_2 \Vert _{q}
 =\Vert G \Vert \left(  \displaystyle \int_{\tau}^{t} \Vert w (\xi) \Vert ^q_{\xi} d\xi\right)^{\frac{1}{q}}\\
 &\leq \Vert G \Vert \left[ \left(  \displaystyle \int_{\tau}^{t} \Vert w (\xi) \Vert ^p_{\xi} d\xi\right)^{\frac{q}{p}}(t-\tau)^{1-\frac{q}{p}}\right]^{\frac{1}{q}}\\
 &=\Vert G \Vert (t-\tau)^{\frac{1}{q}-\frac{1}{p}} \left(  \displaystyle \int_{\tau}^{t} \Vert w (\xi) \Vert ^p_{\xi} d\xi\right)^{\frac{1}{p}}\\
 &=\Vert G \Vert (t-\tau)^{\frac{1}{q}-\frac{1}{p}} \left(  \displaystyle \int_{\tau}^{t} \Vert \psi (\xi) w (\xi) \Vert ^p_{\xi} d\xi\right)^{\frac{1}{p}}\\
 &\leq \Vert G \Vert (t-\tau)^{\frac{1}{q}-\frac{1}{p}} \Vert \psi w \Vert _{p}\\
 &\leq 2Ke^c \Vert G \Vert ^2 (t-\tau)^{\frac{1}{q}-\frac{1}{p}} \Vert T(t,\tau)x \Vert _t.
\end{align*}
Hence, $(t-\tau)\Vert x \Vert _{\tau}=(t-\tau)\Vert w (\tau) \Vert _{\tau}=(t-\tau)\|v_2(\tau)\| \le 2Ke^c \Vert G \Vert ^2 (t-\tau)^{\frac{1}{q}-\frac{1}{p}} \Vert T(t,\tau)x \Vert _t$. So,
$$\Vert x \Vert _{\tau} \leq \dfrac{2Ke^{c}\Vert G \Vert ^2}{(t-\tau)^{1-\frac{1}{q}+\frac{1}{p}}}\Vert T(t,\tau)x \Vert _t.$$
Since $(p,q)\ne (\infty,1)$, $1-\frac{1}{q}+\frac{1}{p}>0$. Thus, $T>0$ exists such that \eqref{2.23} holds.

For $t \geq \tau$ and $x\in X$, write $t-\tau =kT+s$ with $0 \leq s <T$ and $k\in \mathbb{N}$. Then,
$$T(t,\tau)Q(\tau)x=\prod_{i=0}^{k-1} T(t -iT, t-(i+1)T)T(\tau+s,\tau)Q(\tau) x.$$
By \eqref{2.22} and \eqref{2.23},
 \begin{align*}
 \Vert T(t,\tau)Q(\tau)x \Vert _t  &\geq 2^k \Vert T(\tau +s,\tau)Q(\tau)x\Vert _{\tau+s}\\
 & \geq \frac{2^k}{C}\Vert Q(\tau)x\Vert _{\tau}
 \geq \frac{1}{2C}e^{\frac{(t-\tau) \ln 2}{T}}\Vert Q(\tau)x\Vert _{\tau}.
 \end{align*}
So, the inequality \eqref{2.21} holds for $\lambda =\frac{\ln 2}{T}$ and $D=2C$.
\end{proof}

To complete the proof, we need to verify the remainder conditions of the exponential dichotomy.
By Lemma \ref{lem3} and Lemma \ref{lem4}, $T(t,\tau): Q(\tau)X \to Q(t)X$ is bijective and
$$ \|T(\tau,t)_|Q(t)x\|_{\tau} \le {D} e^{-\lambda(t-\tau)}\Vert Q(t)x\Vert _t \le D(1+M)e^{-\lambda(t-\tau)}\Vert x\Vert _t$$
for all $t\ge \tau$. By Lemma \ref{lem3}, 
we have $T(t,\tau)Q(\tau)x\in F^u_{t}$ and $T(t,\tau)P(\tau)x\in F^s_{t}$. Therefore,
$P(t)T(t,\tau)x=T(t,\tau)P(\tau)x$ for every $x\in X$, which means that $P(t)T(t,\tau)=T(t,\tau)P(\tau)$ for $t\ge \tau$. So, the evolutionary family $(T(t,\tau))_{t \geq \tau}$ is exponentially dichotomic with respect to the family of norms $\Vert \cdot\Vert _t$.
\end{proof}
\begin{commen}\label{re2}
By Remark \ref{re1} and the definition of the subspaces $F^s_{\tau}, F^{u}_{\tau}$, we obtain $X^s(\tau)\subset F^s_{\tau}$ and $X^u(\tau)\subset F^u_{\tau}$ for each $\tau \in \r$. Therefore, by Lemma \ref{lem} and  Remark \ref{re1} we have $X^s(\tau) = F^s_{\tau}$ and $X^u(\tau) = F^u_{\tau}$ for each $\tau \in \r$. So, the family of dichotomic projections is unique.
\end{commen}

As an application, we now provide a result for the perturbation.
\begin{theorem}
Let the evolution family $(T(t,\tau))_{t\geq \tau}$ be exponentially dichotomic with respect to a family of norms $\Vert \cdot\Vert _t$ and  $B: \mathbb{R} \to B(X)$ be a strong continuous function such that 
\begin{equation}
\Vert B(t)\Vert \leq Me^{-\epsilon \vert t\vert}\varphi(t) ,\quad t \in \mathbb{R}. \label{2.24}
\end{equation}
Then, for $\varphi \in L^q(\r)$ and sufficiently small $M>0$, the evolution family $U(t,\tau)$ satisfying
$$U(t,\tau)=T(t,\tau)+ \displaystyle \int_{\tau}^t T(t,s)B(s)U(s,\tau)ds,\quad  t \geq \tau \text{ and}\quad t,\tau \in \mathbb{R} $$
admits an exponential dichotomy with respect to the same family of norms.
\end{theorem}
\begin{proof}
Let $L$ be the linear operator associated to the evolution family  $(U(t,\tau))_{t \geq \tau}$, in the same as the linear operator $H$ is associated to the evolution family  $(T(t,\tau))_{t\geq \tau}$, see Definition \ref{H}. More precisely, $L$ is defined by $Lx=y$ in the domain $\mathcal{D}(L)$, namely
$$x(t)=U(t,\tau)x(\tau)+ \displaystyle \int_{\tau}^t U(t,s)y(s)ds,\quad  t \geq \tau.$$
By direct transformation and using Fubini theorem, we obtain the following relation
\begin{align}
x(t)=T(t,\tau)x(\tau)+\displaystyle \int_{\tau}^t T(t,w) \left( y(w)+B(w)x(w) \right) dw \quad  t \geq \tau. \label{2.25}
\end{align}
We now define an operator $P: Y_1 \to Y_2$ by $(Px)(t)=B(t)x(t)$ for $t \in\mathbb{R}$. By \eqref{2.3} and \eqref{2.24}, we have
\begin{align*}
\Vert B(t)x(t) \Vert _t &\leq Ce^{\epsilon\vert t \vert} \Vert B(t)x(t) \Vert 
\leq MC \varphi(t) \Vert x(t) \Vert 
\leq MC \varphi(t) \Vert x(t) \Vert  _t \le MC \varphi(t) \|x\|_{Y_1}
\end{align*}
for $t \in \mathbb{R}$. Therefore, we have
\begin{align*}
\Vert Px\Vert _q=\Big(\int_{-\infty}^{+\infty} \Vert B(t)x(t) \Vert _t^q dt\Big)^{\frac{1}{q}} \leq MC \|x\|_{Y_1} \Big( \int_{-\infty}^{+\infty} (\varphi(t))^q dt \Big)^{\frac{1}{q}} = MC \|\varphi\|_q \|x\|_{Y_1}.
\end{align*}
Hence, $P$ is a bounded linear operator and 
\begin{equation}
\Vert P\Vert  \leq MC \|\varphi\|_q. \label{2.26}
\end{equation}
By the assumption, we have $H:\mathcal{D}(H) \rightarrow Y_2 $ has a bounded invertible operator.
By \eqref{2.26}, for $M>0$ is small enough, we obtain
$$\Vert PH^{-1} \Vert \leq \Vert P\Vert \Vert H^{-1} \Vert \leq MC  \|\varphi\|_q \Vert H^{-1} \Vert <1.$$
Therefore, $I-PH^{-1}: Y_2 \to Y_2$ is invertible. 
Furthermore, it follows from \eqref{2.25} that $\mathcal{D}(H)=\mathcal{D}(L)$ and $H=L+P$. So,
$$ L=H-P=(I-PH^{-1})H: \mathcal{D}(H) \rightarrow Y_2$$
is also invertible for sufficiently small $M>0$.

It remains to show that there exists $K_1, c_1>0$ such that
\begin{equation*}
\Vert U(t,\tau)x\Vert _t \leq K_1e^{c_1(t-\tau)}\Vert x\Vert _{\tau} \quad \text{ for} \quad t\geq \tau. \label{2.27}
\end{equation*}
Indeed, we have
\begin{align*}
\Vert U(t,\tau)x \Vert _t &=\Big\Vert T(t,\tau)x + \displaystyle \int_{\tau}^t T(t,s)B(s)U(s,\tau)xds\Big\Vert _t\\
&\leq Ke^{c(t-\tau)}\Vert x\Vert _{\tau} + K \displaystyle \int_{\tau}^t e^{c(t-s)}\Vert B(s)U(s,\tau)x\Vert _sds\\
&\leq  Ke^{c(t-\tau)}\Vert x\Vert _{\tau} + MCK \displaystyle \int_{\tau}^t e^{c(t-s)} \varphi(s) \Vert U(s,\tau)x\Vert _sds
\end{align*}
for $t \geq \tau$. Let $\phi (t)=e^{-ct}\Vert U(t,\tau)x\Vert _t$, we get
$$ \phi (t) \leq K\phi (\tau)+MCK \displaystyle \int_{\tau}^t  \varphi(s) \phi (s)ds. $$
Using Gronwall's lemma, we have
$$ \phi (t) \leq K\phi(\tau)e^{MCK  \displaystyle\int_{\tau}^t  \varphi(s) ds}.$$
On the other hand,
$$ \int_{\tau}^t  \varphi(s) ds \le \sum_{k=0}^{[t-\tau]}\int_{\tau+k}^{\tau +k+1}\varphi(s) ds \le \|\varphi\|_q (t-\tau +1).$$
Thus,
$$ \phi (t) \leq K\phi(\tau)e^{MCK \|\varphi\|_q (t-\tau +1) }. $$
This implies that
$$ \Vert U(t,\tau)x\Vert _t \le  Ke^{MCK\|\varphi\|_q}e^{(c+MCK \|\varphi\|_q)(t-\tau)}  \Vert x\Vert _{\tau} \quad \text{ for} \quad t\geq \tau. $$
Applying Theorem \ref{thm:2.3}, the evolution family  $(U(t,\tau))_{t\ge \tau}$ is exponentially dichotomic with respect to the family of norms $\Vert \cdot\Vert _t$.
\end{proof}

\bibliographystyle{plain}

\end{document}